\documentclass[12pt,reqno]{amsart}
\usepackage[active]{srcltx}
\usepackage{amscd,amssymb,amsthm,euscript}
\newtheorem{thm}{Theorem}[section]

\newtheorem{prop}[thm]{Proposition}

\theoremstyle{definition}
\newtheorem{rem}{Remark}[thm]

\DeclareMathOperator{\interior}{\ensuremath{\mathop{int}}}

\newcommand{\h}{\mathcal{H}}

\newcommand{\B}{\mathcal{B}}

\newcommand{\R}{\mathbb{R}}
\newcommand{\F}{\mathcal{F}}

\newcommand{\G}{\mathcal{G}}
\newcommand{\Q}{\mathbb{Q}} 

\newcommand{\LL}{\mathcal{L}}

\newcommand{\x}{\times}
\newcommand{\iso}{\cong}
\newcommand{\Z}{\mathbb{Z}}
\newcommand{\cont}{\subseteq}

\newcommand{\dx}{ \{ D_x \} }
\newcommand{\twohat}[1] {\hat{\hat{#1}}} 

\newcommand{\te}{\otimes}

\begin{document}

\title{ Spectral multiplicity and odd K-theory-II}

\author{Ronald G. Douglas}
\address{Department of Mathematics\\
Texas A\&M University\\
College Station, TX 77843-3368}
\email{rdouglas@math.tamu.edu}
\author{Jerome Kaminker}
\address{Department of Mathematics\\UC Davis\\Davis, CA 95616}
\email{kaminker@math.ucdavis.edu}
\date{\today}

\begin{abstract}
Let $\dx$ be a family of unbounded self-adjoint Fredholm operators representing an element of $K^1(M)$. Consider the first two components of the Chern character.  It is known that these correspond to the spectral flow of the family and the index gerbe.  In this paper we consider descriptions of these classes, both of which are in the spirit of holonomy.  These are then studied for families parametrized by a closed 3-manifold.  A connection between the multiplicity of the spectrum (and how it varies) and these classes is developed.   

\end{abstract}
\maketitle

\section{Introduction}

In a previous paper, \cite{Douglas-Kaminker:1}, we studied the set of unbounded, self-adjoint, Fredholm operators with compact resolvent, filtered by the maximum dimension of eigenspaces.  One goal was to relate the vanishing of the components of the Chern character of the K-theory class determined by a family of such operators to its maximum filtration.  In the present note we will give a simple uniform description of the first two terms in the Chern character and study them in the case of certain families over a closed 3-manifold.  

\section{Invariants of families of operators}

We will consider families of operators $\dx$ where $x \in M$, $M$ a closed smooth manifold.  The operators will be unbounded, self-adjoint, Fredholm operators with compact resolvent. We follow the notation and ideas of the paper, \cite{Douglas-Kaminker:1}, and refer there for details.  The (spectral) graph of the family is the closed subset $\Gamma(\dx) = \{(x,\lambda) | \lambda \in sp(D_x) \} \cont  M \x \R$.  If the graph is not connected, then the family determines the trivial element in $K^1 (M)$.  In \cite{Douglas-Kaminker:1}, the approach taken was to adapt topological obstruction theory to determine when one deform the family to one with a disconnected graph.  The obstructions met along the way were related to the components of the Chern character of the K-theoy class defined by the family. Under the assumption that the multiplicity of the spectrum of the family was bounded by 2, it was determined that the first obstruction corresponded to spectral flow and the second to the index gerbe.  In the next two subsections we will give elementary definitions of these classes.  To this end, we note that there is a finite cover of the parameter space $M$, $\{U_i\}$, $i=1,\ldots,N$, and real numbers $\lambda_i$, such that $\lambda_i$ is not in the spectrum of $D_x$, for $x \in U_{i}$. Moreover, we can and will assume that the sets in the cover and their finite intersections are contractible. 
We will build the invariants relative to this data. 

\subsection{Spectral flow}

We will define a \v{C}ech cocycle in $Z^1(\{U_i\},\underline{\Z})$.  Fix an open set $U_i$ in the cover and, for $x \in U_i$, let $P_{U_i} (x)$ be the orthogonal projection onto the subspace of $\h$ spanned by the  eigenspaces of $D_x$ for eigenvalues greater than $\lambda_i$.  Note that $P_{U_i}(x)$ varies continuously in $x$.  Suppose that $U_j$ is another element of the cover and that $x \in U_i \cap U_j$.  If $\lambda_i > \lambda_j$, then  $P_{U_i}(x) - P_{U_j}(x)$ is a finite rank projection.  so that one may associate to the (ordered) pair of projections their codimension, $\dim (P_{U_i}(x), P_{U_j}(x)) \in \Z$. However, for simplicity in stating and proving the next Proposition, we will use the notion of essential codimension, $ec(P_{U_i}(x), P_{U_j}(x))$, which agrees with the usual notion of codimension in the case above.  We refer to \cite{BDF77} for the definitions and properties. To this end, we define $c(U_i,U_j)(x) = ec(P_{U_i}(x), P_{U_j}(x))$.  Since $U_i \cap U_j$ is connected, the value is independent of $x$.  This leads to the following result.
\begin{prop}
 The function $c$ is a 1-dimensional cocycle in $Z^1(\{U_l\},\underline{\Z})$.  Its image,  $sf(\dx) = [c] \in H^1(M,\Z) \in H^1(M,\Q)$  is equal to the spectral flow class of the family $\dx$, which is a rational multiple of the first component of the Chern character of the K-theory class represented by the family.
\end{prop}
\begin{proof}
 Using the definition of spectral flow described in \cite{Douglas-Kaminker:1}, the result follows by a direct modification.   
\end{proof}

\subsection{Index gerbe}

We will construct a 3-dimensional integral cohomology class on $M$ and show that it agrees with the Dixmier-Douady invariant of the index gerbe.  This is based on the work of Carey, Mickelsson and others, \cite{Lott:2002,Carey:1, Carey-Mickelsson:1, Mickelsson:1}.  We again start with the family of projections defined over the sets in the open cover, $\{U_i \}$.  Fix a point $x \in U_i$.  The projection $P_{U_i}(x)$ determines a quasi-free representation of the CAR algebra, $ \alpha_{U_i}(x) : CAR(\h) \to \B (\F_{P_{U_i}(x)}(\h))$, where $\F_{P_{U_i}(x)}(\h) = \F(P_{U_i}(x)\h) \te \F(\overline{(I-P_{U_i}(x))\h})$ is the Fermionic Fock space. As a reference to that theory we refer to \cite{Araki:1}. If $x \in {U_i} \cap {U_j}$, then we have two different representations, $\alpha_{U_i}(x)$ and   
$\alpha_{U_j}(x)$, which are equivalent.  This follows since the projections yielding them differ by a finite rank projection and one can construct a canonical intertwining unitary operator, (see Appendix), $$S_{P_{U_j}(x),P_{U_i}(x)} : \F_{P_{U_j}(x)}(\h) \to \F_{P_{U_i}(x)}(\h),$$  satisfying
\begin{equation}
 S_{P_{U_j}(x),P_{U_i}(x)} \alpha_{U_i}(x) S_{P_{U_j}(x),P_{U_i}(x)}^* = \alpha_{U_j}(x).
\end{equation}
Note that $S_{P_{U_j}(x),P_{U_i}(x)}$ is defined only up to a choice of a scalar $z \in S^1$.  This ambiguity disappears in the following,
\begin{equation}
 \hat S_{P_{U_j}(x),P_{U_i}(x)} = Ad_{S_{P_{U_j}(x),P_{U_i}(x)}} : \B(\F_{P_{U_j}(x)}(\h)) \to \B(\F_{P_{U_i}(x)}(\h)).
\end{equation}

Next, suppose that $x \in {U_i} \cap {U_j} \cap {U_k}$, for sets ${U_i},{U_j},{U_k}$ in the cover.  Then there exists a function $g(U_i , U_j , U_k) : U_i \cap U_j \cap U_k \to S^1$ such that                                                                                                                                                                                                                                                                  one has 
\begin{equation}
 \hat S_{P_{U_j}(x),P_{U_i}(x)}\circ \hat S_{P_{U_i}(x),P_{U_k}(x)}\circ \hat S_{P_{U_k}(x),P_{U_j}(x)} = g({U_i},{U_j},{U_k})(x) \in S^1,
\end{equation}
since the composition on the left intertwines an irreducible representation of $CAR(\h)$, and hence is a scalar multiple of the identity.  We will show that $g({U_i},{U_j},{U_k})(x)$ can be defined so as to be a continuous function of $x$.  To see this, note first that the projections onto the finite dimensional differences, $\h_{i,j}(x)$, are norm continuous.  We claim there is a continuous family of unitaries, $R_{U_i, U_j}(x)$ with the property that $R_{U_i, U_j}(x)(\h_{i,j}(x)) = \h_{i,j}(x_0)$.  Choosing an orthonormal basis for $\h_{i,j}(x_0)$, a volume element can be determined in a continuous manner for each $\h_{i,j}(x)$.  This is what is needed to define the intertwining unitaries in a continuous way. See the appendix for more on this matter. We will be using sheaf cohomology with respect to the presheaf of $S^1$-valued functions on $M$. c.f. \cite{Brylinski:1}

\begin{prop}
  The family, $\{g({U_i},{U_j},{U_k})\}$ defines a cocycle in $C^2 (\{{U_i}\},\underline{S}^1)$ and, hence, a cohomology class  $[g] \in H^2(\{{U_i}\}, \underline{S}^1)$.
\end{prop}
\begin{proof}
 This is a direct computation.
\end{proof}

There is a natural map $H^2(\{{U_i}\}, \underline{S}^1) \to H^3(M,\Z)$, which is an isomorphism if the cover is well chosen.  We will denote the image of $[g]$ by $\G(\dx) \in H^3(M,\Z)$.  One can check that it depends on the family $\dx$ only up to homotopy.

\begin{thm}
 The class $\G(\dx)$ is equal to the Dixmier-Douady invariant of the Index Gerbe, (as, e.g. defined by Lott, \cite{Lott:2002}).
\end{thm}
\begin{proof}
 Recall that Lott uses Hitchin's definition of gerbe, \cite{Hitchin:1}.  Thus, there are line bundles $\LL_{U_i, U_j}$ over $U_i \cap U_j$  possessing the necessary properties.  It is easily seen that tensoring the bundle of Fock spaces over $U_i \cap U_j$ with $\LL_{U_i, U_j}$ is the same as applying the family of unitary operators $S_{P_{U_j}(x),P_{U_i}(x)}$, (see appendix).  The result follows easily from this observation.  
\end{proof}

\section{A low dimensional example}

We make some preliminary observations.  First we recall the addition operation for families.  If $\dx$ and $\{ D'_x \} $ are families on the Hilbert spaces $\h$ and $\h'$ respectively, then the sum of the families is the direct sum on the Hilbert space $\h \oplus \h'$.  The graph of the sum is the union of the graphs of each summand, and the multiplicity of eigenvalues is the sum of the multiplicities on the intersection of the graphs and the multiplicity for the individual families on the symmetric difference of the graphs.  The inverse of a family is obtained by reversing the signs of the eigenvalues, or by replacing $\dx$ by $-\dx$.

Let $\dx$ be a family parametrized by a smooth, closed, connected manifold,  $M$.  Consider $sf(\dx) \in H^1(M,\Z) = [M,S^1]$.  Let $\alpha : M \to S^1$ represent this class.  Let $\{D'_z\}$ be a family on $S^1$ with spectral flow equal to 1.  Then $\dx - \alpha^* (\{D'_z\})$ has spectral flow zero and it can be analyzed by the methods of \cite{Douglas-Kaminker:1}.  One would like to do an analogous construction to eliminate the index gerbe class in $H^3(M,\Z)$, but because of possible torsion this may not be possible.  However, if $M$ is a closed 3-manifold, then this reduction can be done and we will analyze this case in the present section.

Assume we are given a family, $\dx$, parametrized by a smooth closed 3-dimensional manifold $M$. Let  $\Gamma = \Gamma (\dx) \cont M \x \R$ with projection $\pi : \Gamma \to M$ be the graph.  We will assume that $sf(\dx) = 0$, so that we have a spectral decomposition,
\begin{equation}
 \Gamma = \bigcup_{k\in \Z} G_k.
\end{equation}
We will first impose conditions on the intersections of the sets, $G_k$, of the decomposition.  At the end we show how these conditions can be relaxed in order to obtain the desired results.  To this end, we make the following assumptions. 
\begin{itemize}
 \item[A)] $G_0$ can have a non-empty intersection with  $G_{i}$ only if $i= 1$ or $i=-1$ and $G_0 \cap G_1 \cap G_{-1} = \emptyset$.
\end{itemize}
Condition (A) states that the family has multiplicity $\leq 2$ on $G_0$.

Our second assumption is

\begin{itemize} 

\item[B)] there is a closed ball, $B \cont M$, satisfying    $\ \pi(G_0 \cap G_1 ) \cont B$ and $\pi (G_0 \cap G_{-1} ) \cont M \setminus B $. 
\end{itemize}
Note that if $ \pi(G_0 \cap G_1) = M$, and the multiplicity of the family is $\leq 2$ on $G_0$, then the multiplicity is constant on $G_0$, and it was shown in \cite{Douglas-Kaminker:1} that this implies that the family is trivial.  On the other hand, if 
$ \pi(G_0 \cap G_1) = \emptyset$, then the graph is disconnected, and again it was shown in \cite{Douglas-Kaminker:1} that the family will be trivial. Thus, it suffices to consider the case that $ \pi(G_0 \cap G_1) \neq \emptyset \  \text{and }  \neq M $.

The boundary of  $B$ is a 2-sphere, $S^2$, over which there is a 2-plane bundle, $E$, whose fiber at $x$ is the space of eigenvectors of $g_0(x)$ and $g_1(x)$ in $\h$. This bundle is defined over $B$ and, hence, is trivial.  There is a splitting of $E|_{S^2}$ as the direct sum of two line bundles, $L_0$ and $L_1$.  The splitting is determined by the orthogonal eigenspaces for $g_0(x)$ and $g_1(x)$, repsectively, for $x$ outside, but close to, $B$.  Note that outside of $B$, but near to it, we have $g_0(x) < g_1(x)$. The line bundle $L_0$ over $S^2$ is determined by the homotopy class of its clutching map $ \kappa : S^1 \to U_1$,  and hence by an integer $k$.  Moreover, the integer associated to $L_1$ is $-k$.  Unless $k = 0$, the splitting will not extend over the ball, $B$.

Next, observe that there is a degree one map $c : M \to S^3$ which sends $\interior (B)$ to $S^3 \setminus \{ (0,0,-1) \}$ and $ M \setminus \interior (B) $ to $ (0,0,-1) $.  Let $\{ \partial_z^{(-k)}\}$ denote the Mickelsson family, \cite{Douglas-Kaminker:1}, determined by $-k$,  and  consider the family $\hat{\dx} = \dx +c^* ( \{ \partial_z^{(-k)}\})$ over $M$. We will show that this family is equivalent to a trivial family and, hence, $[\dx]  = c^* ( [\{ \partial_z^{(k)}\}]) \in K^1(M)$.

\begin{prop}
\label{sep1}
 The family $\hat{\dx}$ can be deformed to a family ${\hat{\hat{\dx}}}$ which satisfies $ \hat{\hat{G_0}} \cap \hat{\hat{G_1}} = \emptyset$.  Moreover, ${\hat{\hat{\dx}}}$  is equal to $\dx + c^* ( \{ \partial^{(-k)}_{z}\})$  outside a neighborhood $U \supseteq B$, with $\bar U$ disjoint from $G_0 \cap G_{-1}$.
\end{prop}
\begin{proof}
 To verify this claim we will use the methods of \cite{Douglas-Kaminker:1}. Shrink $B$ to $B' \cont B$.  We may arrange that $g_0(x) = 0$ for $x \in B'$ for $\dx$, and $g'_0(x) = 0$ for  $x \in c(B')$ for $c^* ( \{ \partial^{-k}_{z}\})$.  Now flatten, \cite[p. 322]{Douglas-Kaminker:1}, both $\dx$ and $c^* ( \{ \partial_z^{(-k)}\})$ on $B'$ relative to $B$.  The span of the eigenspaces for $g_0(x),g_1(x), g'_0(x), g'_1(x)$ is a 4-dimensional vector bundle $E$ over $B'$.  More specifically, the graph of the sum of the flattened families has the following structure on $\pi^{-1}(B')\cap G_0$. The multiplicity of each eigenvalue is 4, so there is a (trivial) 4-plane bundle, $E$, over $B$ with a natural splitting on the boundary into a direct sum of four line bundles, $L_0, L_1, L'_0$, and $L'_1$, where $L'_0$, and $L'_1$ are associated with $c^* ( \{ \partial_z^{(-k)}\})$ and $L_0, L_1$ with $\dx$. Now, $L_0$ and $L'_0$ are determined by integers which are negatives of each other.  Thus, the sum $L_0 \oplus L'_0$ is trivial, as is $L_1 \oplus L'_1$.  These facts imply that the splitting $E \iso (L_0 \oplus L'_0) \oplus (L_1 \oplus L'_1)$ extends across $\pi^{-1}(B')$. We may then use the scaling operation from \cite[p. 324]{Douglas-Kaminker:1} to deform the family by increasing the eigenvalue on the subspaces spanned by $(L_1 \oplus L'_1)$ to a small value, $\varepsilon > 0$ on $B'$, decaying to 0 on $B \setminus B'$, and leaving the family unchanged outside $B$ .  One then obtains that $\twohat{G_0} \cap \twohat{G_1} = \emptyset$ as required.      
\end{proof}

Note that a family $\hat{\hat{{\dx}}}$ for which $\twohat{G_{0}} \cap \twohat{G_{1}} = \emptyset  $ satisfies that $ \twohat{g_0}(x) < \twohat{g_1}(x) $ for all $ x \in M $, and hence the family is trivial.


We thus obtain, 

\begin{thm}
 For a family satisfying (A) and (B) above, one has
\begin{equation}
 \dx \simeq c^* ( \{ \partial^{k}_z\} ),
\end{equation}
where $k$ is the degree of the bundle as above.
\end{thm}

At this point, one would like to deform an arbitrary family into one which satisfies both hypotheses, (A) and (B).  While one can obtain (B), assuming (A), it is not so easy to establish (A) in the generality in which we are working.  Although we are able to deform a family with multiplicity less than or equal to 3 to one satisfying (A), the general case requires more elaborate techniques which will be addressed later.  Thus, we will thus assume now that (A) holds and describe next how we can use the methods of \cite{Douglas-Kaminker:1} to ensure that the set where the multiplicity is 2, $\pi(G_0 \cap G_1)$, is  contained in a ball whose complement contains $\pi(G_0 \cap G_{-1})$.

\begin{prop}
\label{engulf}
 Given a family satisfying (A), one can alway deform it to a family satisfying condition (B).
\end{prop}
\begin{proof}
We proceed as in \cite[p. 325]{Douglas-Kaminker:1}.  Since the multiplicity is at most 2 on $G_0$, we can find open sets with disjoint closures, $V$ and $W$ contained in $M$, with $\pi(G_0 \cap G_1) \cont W$ and $\pi(G_0 \cap G_{-1}) \cont V$.   Triangulate $M$ finely enough so that any simplex which meets  $\pi(G_0 \cap G_1)$ is contained in $W$ and any which meets  $(G_0 \cap G_{-1})$ is contained in $V$.  Now we proceed, inductively over skeleta, to deform the family so that $G_0$ is separated from $G_1$ and $G_{-1}$, over the 0, 1 and 2 skeletons . 
At this stage, $G_0 \cap G_1$ and $G_0 \cap G_{-1}$ for the deformed families are contained in the interiors of 3-simplices, and hence in sets homeomorphic to 3-balls.  Now, if we remove from $M$ the 3-balls containing $G_0 \cap G_{-1}$, the result is path connected.  Thus we may find paths between the 3-balls containing $G_0 \cap G_1$ and avoiding the 3-balls containing $(G_0 \cap G_{-1})$, and then thicken them to be tubes.  This can be done in such a way that the result is a single closed 3-ball containing   $G_0 \cap G_1$ with $(G_0 \cap G_{-1})$ in its exterior.
\end{proof}

This yields the following theorem.

\begin{thm}
 Any family over a closed connected 3-manifold satisfying (A) is equivalent to one of the form  $\alpha^*(sf) + c^* ( \{ \partial^{k}_z\} )$, where $sf$ is the family over the circle with spectral flow 1, and $\alpha : M \to S^1$ is the map representing $sf(\dx)$. 
\end{thm}
 
\begin{rem}
 This result can be deduced strictly using algebraic topology.  However, we have shown something stronger.  While the ``moves'' we used to pass from the given family to the one in standard form imply homotopy of the families, the converse is not known to be true.  Thus, if we define $\hat K^1(M)$ to be the group generated by families of operators modulo flattening and scaling, there is a surjection
\begin{equation}
 \hat K^1(M) \to K^1(M) \to 0.
\end{equation}
It would be interesting if it were not injective in general.  
\end{rem}
\begin{rem}
The method of studying $K^1(M)$ we use, actually starts with a subset of $M \x \R$ with an eigenspace associated to each point--that is, one has an enhanced graph.  Over sets of constant multiplicity one has a vector bundle.  This could be viewed as a generalized local coefficient system on $M$ and one could then try to define twisted cohomology groups.  A theory along these lines was developed in \cite{Mrowka-Kronheimer:1}.  It would also be useful to consider a notion of concordance, apriori weaker than homotopy, for enhanced graphs and obtain a result stating that the concordance is trivial if a family of generalized characteristic classes agree.
\end{rem}
\begin{rem}
 Suppose we have a family , $\dx$, over a closed, oriented 3-manifold which satisfies (A) and (B).  We can deform as in Proposition \ref{engulf} and proceed further so that $G_0 \cap G_1$ consists of a single point, $\{ x\} $.  Then $k\{x\}$ defines a 0-dimensional homology class which is Poincar\'e dual to the 3-dimensional class $\G(\dx)$.  This is analogous to a result of Cibotaru, \cite{Cibotaru:1}.
\end{rem}

\section{Appendix}

We review the construction and properties of the intertwining operators in the case at hand-- that is, two projections $P,Q$ with $P-Q$ of dimension N, so $P \succ Q$.  Given a Hilbert space $\h$, we will denote by $\bar{\h}$ the conjugate Hilbert space. Decompose $\h$ using $P$ and $Q$,
\begin{equation}
 \h = P\h \oplus (1-P)\h = Q\h \oplus (1-Q)\h.
\end{equation}

Then we have 
\begin{equation}
\begin{split}
              \F_Q (\h)& = \F(Q\h) \te \F(\overline{(P-Q)\h})\te \F(\overline{(I-P)\h})\\
              \F_P (\h)& = \F(Q\h) \te \F((P-Q)\h)\te \F(\overline{(I-P)\h}).
\end{split}
\end{equation}

The isomorphism $S_{Q,P} : \F_P (\h) \to \F_Q (\h)$ is then defined to be 
\begin{equation}
 S_{Q,P} = I_{\F(Q\h)} \te \tilde{S}_{Q,P} \te I_{\F(\overline{(I-P)\h)}},
\end{equation}a non-zero element of the top exterior power
where $\tilde{S}_{P,Q} : \F({(P-Q)\h}) \to \F(\overline{(P-Q)\h})$ is defined as follows.  Choose a volume element $\omega \in \F((P-Q)\h)$, that is, a non-zero element of the top exterior power, $\Lambda^{N}((P-Q)\h)$.  Then $\tilde{S}_{P,Q}$, on $\Lambda^{k}((P-Q)\h)$, is the composition,
\begin{equation}
\label{d1}
\begin{split}
 \Lambda^{k}((P-Q)\h) \xrightarrow{\Theta} \Lambda^{N-k}((P-Q)\h)^* \to \\ \to \Lambda^{N-k}(((P-Q)\h)^*) \to \Lambda^{N-k}(\overline{(P-Q)\h}), 
\end{split}
\end{equation}
where the first map is $\Theta(x)(y) = <x \wedge y , \omega>$ and the latter two are canonical isomorphisms.  Thus, the composite isomorphism depends only on the choice of $\omega$. 


One may view (\ref{d1}) as defining an isomorphism
\begin{equation}
\label{d2}
  \bigoplus_{k=0}^ N \left( \Lambda^{k}((P-Q)\h)\right) \te \Lambda^{N}((P-Q)\h) \to \bigoplus_{k=0}^ N\Lambda^{N-k}(\overline{(P-Q)\h}).
\end{equation}

If the projections, $P$ and $Q$, depend continuously on  $x \in M$, then (\ref{d2}) extends to an isomorphism of bundles,
\begin{equation}
\F((P-Q)\h) \te \mathcal{L}_{P,Q} \to \F(\overline{(P-Q)\h}).
\end{equation}

The map $\tilde{S}_{Q,P}$ is obtained by fixing a non-zero cross-section of $\mathcal {L}_{P,Q}$, and it depends on that choice.  In the cases considered in the paper one has that $\mathcal {L}_{P,Q}$ is trivial.

\providecommand{\bysame}{\leavevmode\hbox to3em{\hrulefill}\thinspace}
\providecommand{\MR}{\relax\ifhmode\unskip\space\fi MR }
\providecommand{\MRhref}[2]{%
  \href{http://www.ams.org/mathscinet-getitem?mr=#1}{#2}
}
\providecommand{\href}[2]{#2}


\end{document}